\documentclass[12pt,reqno]{amsart}
\usepackage{amscd,amssymb,amsmath,amsthm}
\usepackage{mathtools}
\usepackage{enumerate}
\usepackage[english]{babel}
\usepackage[T1]{fontenc}
\usepackage[arrow,matrix]{xy}
\usepackage{graphicx}
\usepackage{tikz}
\usetikzlibrary{positioning,matrix,arrows,calc}
\usepackage{cite}
\usepackage{dsfont}
\newtheorem{theorem}{Theorem}
\newtheorem{lemma}{Lemma}[section]
\newtheorem{proposition}[lemma]{Proposition}
\newtheorem{definition}[lemma]{Definition}
\newtheorem{remark}[lemma]{Remark}

\begin{document}
\title[]{Symmetric homoclinic tangles in reversible dynamical systems have positive topological entropy}

\author{A.J. Homburg, J.S.W. Lamb, D.V. Turaev}

\address{A.J. Homburg\\ KdV Institute for Mathematics, University of Amsterdam, Science park 107, 1098 XG Amsterdam, Netherlands\newline Mathematics Institute, Leiden University, 
Einsteinweg 55,
2333 CC Leiden, Netherlands}
\email{a.j.homburg@uva.nl}

\address{J.S.W. Lamb\\  Department of Mathematics,
	Imperial College London,
	180 Queen's Gate,
	London SW7 2AZ,
	United Kingdom\newline
	International Research Center for Neurointelligence (IRCN), The University of Tokyo, Tokyo 
113-0033, Japan\newline
Centre for Applied Mathematics and Bioinformatics, Department of Mathematics and Natural Sciences, Gulf University for Science and Technology, Halwally, 32093 Kuwait.
	}
\email{jsw.lamb@imperial.ac.uk}

\address{D.V. Turaev\\  Department of Mathematics,
	Imperial College London,
	180 Queen's Gate,
	London SW7 2AZ,
	United Kingdom}
\email{d.turaev@imperial.ac.uk}

\begin{abstract}
We consider reversible vector fields in $\mathbb{R}^{2n}$ such that the set of fixed points of the involutory reversing symmetry is $n$-dimensional. Let such system have  
a smooth one-parameter family of symmetric periodic orbits which is of saddle type in normal directions. We establish that the topological entropy is positive 
when the stable and unstable manifolds of this family of periodic orbits have a strongly-transverse intersection.  
\end{abstract}

\maketitle

\section{Introduction}

It is a classical result in the theory of dynamical systems that homoclinic tangles give rise to
hyperbolic horseshoes and thus positive topological entropy. 
The history of chaotic dynamics started with the discovery by Poincar\'e \cite{poin} that the stable and unstable manifolds of a saddle periodic orbit may have a transverse intersection along a homoclinic orbit.
For a sufficiently small neighbourhood of the union of a hyperbolic periodic orbit and its transverse homoclinic, the invariant set that consists of all orbits that
stay entirely in this neighbourhood is uniformly hyperbolic and admits a symbolic representation by a full shift on two symbols \cite{sma65,shi67}.
This result, the Shilnikov-Smale theorem, provides the most fundamental criterion for chaos in a dynamical system.

The fact that the Poincar\'e's homoclinic tangle implies positive topological entropy holds true also in the  
original Hamiltonian setting.  A subtle point here is that the Hamiltonian function is a first integral, and saddle periodic orbits of a Hamiltonian system arise in families, parameterised by the value of the Hamiltonian. Such family is a normally-hyperbolic invariant manifold; the homoclinic tangle corresponds to an intersection of its stable and unstable manifolds. Formally speaking, each periodic orbit in the family is not hyperbolic. However, inside any dynamically invariant level set of the Hamiltonian, the saddle periodic orbit is isolated and hyperbolic with a transverse homoclinic, so the Shilnikov-Smale theorem is applied and the positivity of the topological entropy follows.

Normally-hyperbolic one-parameter families of periodic orbits with transversely intersecting stable and unstable manifolds also naturally arise in reversible systems\cite{homlam06}. Despite the substantial interest in reversible dynamical systems \cite{lamrob96,baretal19,becetal09,cha98,dev76,gon13,gonetal14,her95,har98,homkno06,homlam06,lamste04,roblam95,
sev86,gt18}, a concise characterisation of a reversible homoclinic tangle, which we believe deserves to be 
central to the theory of chaotic dynamics in reversible systems, has been lacking.

The core issue here is that reversible systems do not need to be Hamiltonian and, typically, there exists no first integral.
For example, if a perturbation of a reversible Hamiltonian system preserves the reversibility but breaks the Hamiltonian structure, then a given family of symmetric periodic orbits and their symmetric homoclinics survives the
perturbation. However, the dynamically invariant foliation by energy levels gets, typically, destroyed, as the energy is no longer conserved. This provides the possibility that many orbits leave a neighbourhood of the homoclinic tangle due to the drift in energy, which makes the dynamics near a reversible homoclinic tangle very much different from those in the Hamiltonian setting.

The a priori non-controllable drift along the central direction means that one should go beyond the standard hyperbolicity techniques to resolve even the most basic question -- whether the dynamics near the reversible tangle are chaotic? In this paper, we give an affirmative answer by proving that the set of orbits that remain in any given neighbourhood of the reversible homoclinic tangle (satisfying transversality conditions) has positive topological entropy (see Theorem~\ref{t:main}).\\
 
Let us describe the setting in detail. We consider the flow of a smooth vector field
\[
\dot{x} = f(x)
\]
on $\mathbb{R}^{2n}$, $n \ge 2$, which is reversible with respect to a smooth involution 
$R:\mathbb{R}^{2n} \to \mathbb{R}^{2n}$ with an $n$-dimensional manifold of fixed points. 
So 
\[
 D_xR f (x) = - f(R x)
\]
and $\mathrm{Fix}\, R = \{ u \in \mathbb{R}^{2n}\; ; \; R u = u\}$ is $n$-dimensional.

Recall that a periodic orbit $\gamma$ is symmetric if $R\gamma = \gamma$ as a set. 
A symmetric periodic orbit $\gamma$ intersects $\mathrm{Fix}\, R$ transversely in two points.
Consider a symmetric $(2n-1)$-dimensional cross-section $S = RS$ through one of the two points
of $\gamma \cap \mathrm{Fix}\, R$.
Since $S=RS$, the intersection $S \cap \mathrm{Fix}\, R$ is $n$-dimensional, so its image
under the first-return map to $S$ is also $n$-dimensional. The dimension count shows  
that transverse intersections of this image with $S \cap \mathrm{Fix}\, R$ occur along
one-dimensional curves, hence symmetric periodic orbits generically arise in one-parameter families \cite{dev76}. 

Thus, we consider a one-parameter family of symmetric periodic orbits
$\gamma_a = R\gamma_a$ which are of saddle type in the normal directions, parameterised by a real parameter $a$ that runs some interval $(a_-,a_+)$ that contains zero.
When we use the parameter $a$ in our notation, it is understood refer to such an interval. 
This family is a normally-hyperbolic invariant manifold (a two-dimensional cylinder)
for the flow of $f$.
Let $V^s( \gamma_a)$ denote the $n$-dimensional stable manifold of the periodic orbit $\gamma_a$ and $V^u(\gamma_a)$ its $n$-dimensional unstable manifold. By the reversibility, $V^u(\gamma_a)= R V^s( \gamma_a)$.
Assume that $V^{u} (\gamma_0)$ {\em has a transverse intersection with} $\mathrm{Fix}\, R$ at a point $q_0$
outside of $\gamma_0$. By symmetry, $V^{s}(\gamma_0)$ has a transverse intersection with $\mathrm{Fix}\, R$ at the same
point. The invariant manifolds $V^{s}(\gamma_0)$ and $V^{u}(\gamma_0)$ consist of the whole orbits of
the system $\dot x=f(x)$, so the tangents to $V^{s}(\gamma_0)$ and $V^{u}(\gamma_0)$ at the point $q_0$ both contain the
vector $f(q_0)$. We assume, as a non-degeneracy condition, that {\em the intersection of the tangents contains no other directions}.

The orbit $\eta_0$ of $q_0$ lies both in $V^{s}(\gamma_0)$ and $V^{u}(\gamma_0)$, so it tends to $\gamma_0$ both in forward and backward time, i.e., it is homoclinic to $\gamma_0$. Since it is an orbit of a point in $\mathrm{Fix}\, R$, it is symmetric, i.e.~$\eta_0=R\eta_0$. When the above transversality assumptions hold, we call $\eta_0$ a {\em strongly transverse} symmetric homoclinic connection. Since compact parts of stable and unstable manifolds of $\gamma_a$ vary smoothly with $a$, we obtain a one-parameter family of symmetric homoclinic connections $\eta_a = R \eta_a$ in $V^{u} (\gamma_a) \cap V^{s} (\gamma_a)$, defined for all $a$ in some neighbourhood $(a_-,a_+)$ of $a=0$, constituting a {\em symmetric homoclinic tangle}.  See Figure~\ref{f:homoclinic} for a sketch (in Section~\ref{s:pte} we work with different local cross-sections  containing symmetric periodic points and symmetric homoclinic points). 
\begin{figure}[ht]
	\begin{center}
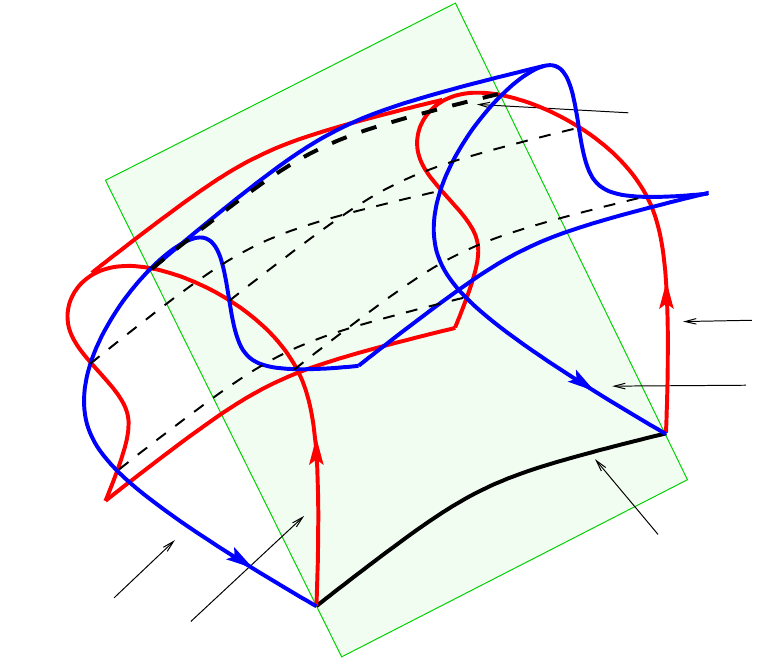
	\end{center}
	\caption[]{Sketch of a homoclinic tangle, here depicted inside a three-dimensional cross-section $S$. 
The family $\{\gamma_a\}$ 
of symmetric periodic orbits, with $a$ from some interval $(a_-,a_+)$,  intersects $\mathrm{Fix}\,R$ inside  $S$, drawn as a solid curve.  Each symmetric periodic orbit $\gamma_a$ has stable and unstable manifolds  $V^s( \gamma_a)$, $V^u(\gamma_a)$. The intersections of these stable and unstable manifolds for the family $\{\gamma_a\}$ with $S$ form two-dimensional sheets: a stable sheet $V^s(\{\gamma_+a\}) \cap S$ bounded by  
 $V^s( \gamma_{a_-})\cap S$ and $V^s( \gamma_{a_+})\cap S$, an unstable sheet $V^u(\{\gamma_+a\}) \cap S$ bounded by  
 $V^u( \gamma_{a_-})\cap S$ and $V^u( \gamma_{a_+})\cap S$.
The sheets  $V^s( \{\gamma_a\})$, $V^u(\{\gamma_a\})$ intersect in $\mathrm{Fix}\,R$ inside $S$, not only in the family of periodic orbits, but also in a family of symmetric homoclinic points $\{\eta_a\} \cap S$. There are further intersections of the stable and unstable sheets outside  $\mathrm{Fix}\,R$, a few are represented by dashed curves.
%
%
		\label{f:homoclinic}}
\end{figure}

Our main result is
\begin{theorem}\label{t:main}
	Consider the flow of a smooth reversible vector field on $\mathbb{R}^{2n}$ that is reversible with respect to an involution
	with $n$-dimensional fixed point manifold.
	Moreover, let there be a normally hyperbolic family of symmetric periodic orbits whose
	stable and unstable manifolds intersect in a family of strongly transverse symmetric homoclinic orbits. Let $\gamma_0$ be a symmetric periodic orbit from this
	family and $\eta_0$ the corresponding symmetric homoclinic
	orbit to $\gamma_0$. Then the flow of $f$ restricted to  any neighbourhood of the union of $\eta_0$ and $\gamma_0$ has positive topological entropy. 
\end{theorem}	

The proof of this result entails the study of skew product systems of interval diffeomorphisms over a horseshoe.
We reduce to a study of the recurrent set of a return map. This set is contained \cite{homlam06} in a lamination
homeomorphic to $\Lambda \times J$ for a horseshoe $\Lambda$ and an interval $J = (-1,1)$.
The dynamics on the lamination is topologically conjugate to a skew product map $F (x,y) = (f(x),  g_x(y))$ of interval diffeomorphisms $g_x$ over a horseshoe map $f$ (in restriction to $\Lambda$, the map $f$ acts as the shift operator 
on the space of 2-symbol sequences). The reversibility of the system translates to the existence of an involution $\hat R$
of the $x$-space such that $\hat R(\Lambda)=\Lambda$,
$f\circ \hat R=\hat R\circ f^{-1}$ and $g_x(y)=g_{\hat R(x)}(y)^{-1}$.
For a Hamiltonian system the derived skew-product system would be
$(x,y) \mapsto (f(x), y)$ and the recurrent set a one-parameter family of horseshoes. For reversible systems, 
the fiber maps of the skew product system need not be the identity map. 
In trivial examples of skew product maps such as $(x,y) \mapsto (x, y + \varepsilon)$ all orbits would drift 
outside of $\Lambda \times J$ and the recurrent set would be empty. However, such example cannot be reversible:  we show that reversibility implies that {\em enough orbits always stay inside} $\Lambda \times J$ to yield positive topological entropy.

Note that we do not establish the existence of finite-type shift dynamics (i.e., a topological conjugacy or semi-conjugacy to a non-trivial Markov chain on an invariant subset of $\Lambda\times J$) which are often associated with positive topological entropy. 
The emergence of shift dynamics after a $C^1$-small perturbation was established in \cite{homlam06}. However, finite-type shift dynamics does not always exist here in spite of the positivity of topological entropy. In particular, it is easy to construct an example of a $C^\infty$-smooth reversible skew-product map $F (x,y) = (f(x),  g_x(y))$ such that for every periodic fiber whose orbit by $F$ is not $\hat R$-symmetric there is a non-zero drift in $y$. This means no non-symmetric periodic orbits in $\Lambda\times J$, which implies no semi-conjugacy to a non-trivial Markov chain in such example.
 
We prove Theorem \ref{t:main} by showing that if $\Lambda \times J$ does not contain an invariant subset where 
the skew-product map $F$ generated by a sufficiently smooth reversible flow is semi-conjugate to a Markov chain
with positive topological entropy, then there always exists a periodic fiber for which the drift in $y$ is positive with a margin of safety,
and, by reversibility, a periodic fiber for which the drift in $y$ is negative. In the spirit of \cite{TZ10} (where the positivity of
space-time entropy for a class of PDEs was proven) we show that the latter case also corresponds to positive entropy.

As an example of an application of Theorem~\ref{t:main}, we mention that symmetric homoclinic tangles of the type we consider may arise locally near homoclinic loops to symmetric equilibria. This includes homoclinic bellows \cite{homkno06} and a homoclinic
loop to a saddle-focus~\cite{baretal19,har98} -- in both cases there exists a symmetric homoclinic tangle, so Theorem~\ref{t:main} implies the positivity of the topological entropy. Note that for the Poincar\'e map of a reversible flow near a symmetric homoclinic loop to a saddle-focus, the positivity of the topological entropy
(and a semi-conjugacy to shift dynamics) is proven in \cite{baretal19}. However, the return time is not bounded for this Poincar\'e map, so proving the positive topological entropy for the flow itself requires Theorem~\ref{t:main} in this situation.

Non-Hamiltonian reversible vector fields with symmetric homoclinic tangles arise in the study of pattern formation in certain classes of partial differential equations \cite{cha98,san02}. For example, for the partial differential equations of the reaction-diffusion type
\[
 u_t =A u_{xx} + N(u), \;   x\in \mathbb{R},   \; u\in U,
\]
where $U$ is an appropriate Banach space of functions $u$ depending on $x\in\mathbb{R}$, a stationary solution
satisfies the ODE 
 \[
 u''(x) = - A^{-1} N(u).
 \]
This equation is invariant under the transformation $x \mapsto -x$, i.e., it is reversible. The involution $R$ acts as 
$u'\to -u'$; its set of fixed points is $\{u'=0\}$ and is half of the dimension of the phase space of the ODE (the space of pairs 
$(u,u')$). Thus, Theorem~\ref{t:main} is applicable. It provides a characterisation of the complexity of the set of solutions near a family of reflection-symmetric solutions that are asymptotically spatially periodic with a localized "defect": the number of different patterns that materialise in a finite spatial window grows exponentially with the window's size. 

Another natural setting of non-Hamiltonian reversible dynamical systems where Theorem~\ref{t:main} may be applied, is that of mechanical systems with non-holonomic constraints. If the system is defined by a Lagrangian $L(\mathbf{q},\mathbf{\dot q})$ with a single constraint $\mathbf{a}(\mathbf{q})\mathbf{\cdot \dot q} =0$, then the equations of motion derived from the
d'Alembert principle are
$$\frac{d}{dt} \frac{\partial L}{\partial \mathbf{\dot q}} - \frac{\partial L}{\partial \mathbf{q}} = 
\mu(t) \mathbf{a}(\mathbf{q}),$$
where the factor $\mu$ is such that the equations are consistent with the constraint at each moment of time. This system preserves the energy $E= L - \frac{\partial L}{\partial \mathbf{\dot q}}\mathbf{\cdot \dot q} $, but it is not Hamiltonian in general (e.g., the phase volume does not need to be preserved). However, when the Lagrangian $L$ is an even function of the velocity vector $\mathbf{\dot q}$, the imposition of the constraint keeps the reversibility in tact. If the space of coordinates 
$\mathbf{q}$ is $(n+1)$-dimensional, then we have $(n+1)$ coordinates and $(n+1)$ velocity components subject to 2 constraints -- the velocity constraint and the energy constraint. Thus, the dimension of the phase space for the system at a fixed energy level is $2n$. The  set of the fixed points of the involution $R: \mathbf{\dot q} \to - \mathbf{\dot q}$ is given by the equation $\{\mathbf{\dot q}=0, \; L(\mathbf{q},0)= E\}$ and has dimension $n$ (i.e., half of the dimension of the phase space) if the energy is in the range of values of $L(\mathbf{q},0)$. One concludes that generic reversible Lagrangian systems with one velocity constraint fall in the class we consider in this paper. An example where Theorem~\ref{t:main} may be applicable
is given by a Chaplygin sleigh \cite{car,chap} moving on a generic surface.

If a non-holonomic mechanical system is symmetric with respect to a continuous group acting on the configuration space, the symmetry reduction decreases the dimension of the configuration space and, hence, the dimension of $\mathrm{Fix}\,(R)$, as
one can see in the examples of rolling spheres \cite{her95} and rattlebacks \cite{gon13}. Adding more velocity constraints increases the dimension of $\mathrm{Fix}\,(R)$ relative to the dimension of the phase space. Thus, one obtains examples of mechanical systems where 
$\mathrm{dim}\, \mathrm{Fix}\,(R)$ is strictly less or greater than half the phase space dimension. In the latter case, the symmetric periodic orbits go in families that depend on more than one parameter. The question of whether symmetric homoclinic
tangles involving such families of periodic orbits {\em always} yield positive topological entropy remains open.

\section{Positive topological entropy}\label{s:pte}

The proof of Theorem~\ref{t:main} involves the study of dynamics of a return map on a cross section. 
It relies on technical results about invariant laminations, reductions to skew product dynamics and estimates on the dynamics in fibers of the skew product system. In this section we prove  Theorem~\ref{t:main}, with reference to auxiliary technical results in later sections.  
 
Consider a symmetric point $p_0$ in $\gamma_0 \cap \mathrm{Fix}\, R$ and a symmetric cross section
$S_0 = R S_0$ containing $p_0$. The family of symmetric periodic solutions
$\{\gamma_a\}$ forms a sheet that intersects $S_0$ in a curve of points $p_a$.

Let $\mathcal{U}$ be a small neighbourhood of $\gamma_0 \cup \eta_0$.  We only consider orbits inside $\mathcal{U}$.
Consider a second symmetric local cross section $S_1 = R S_1$  through the point $q_0$ which is the intersection of  $\eta_0$ with $\mathrm{Fix}\,R$. 
Write $q_a = \eta_a \cap \mathrm{Fix}\,R$ in $S_1$
and let $S = S_0\cup S_1$. A first return map $\Pi: S \to S$ is defined as 
$\Pi (x) = \varphi_t (x)$ with $t>0$ the smallest positive number
for which $\varphi_t (x) \in S$ such that $\varphi_s (x) \in \mathcal{U}$
for all $0\le s \le t$.

We write $\Pi : S \to S$ even though  $\Pi$ is defined only on an open subset of $S$
that has components in  both $S_0$ and $S_1$. 
Restricting the flow to a small neighbourhood $\mathcal{U}$ of  $\gamma_0 \cup \eta_0$ also
means that we consider the families $\{ p_a\}$ and $\{ q_a\}$ for values of $a$ from a small neighbourhood of $0$. 

We restate Theorem~\ref{t:main} in terms of the return map $\Pi:S \to S$.

\begin{theorem}\label{t:mainpi}
The first return map $\Pi:S \to S$	has positive topological entropy.
\end{theorem}

As the return time is bounded, this result implies the positivity of the topological entropy for the flow, hence 
Theorem~\ref{t:main}. The remainder of this section is concerned with the proof of  Theorem~\ref{t:mainpi}, in several steps.

\subsection{Invariant laminations}

The geometry of the invariant set of $\Pi$ in $S$ is clarified by the existence of an invariant lamination $\mathcal{F}^c$:
the lamination consists of leaves that are diffeomorphic to intervals. There are leaves in $S_0$ close to  $\{p_a\}$
and in $S_1$ close to $\{q_a\}$.
These leaves are invariant under the application of an iterate of $\Pi$.
 We let $\mathcal{F}^c_x$ denote the leaf containing the point $x$. 
  The invariant lamination provides a reduction to a skew product of interval maps over a shift. Proposition~\ref{p:skew}
  contains the key result that makes this precise. Its proof is deferred to Section~\ref{s:clsp}. Before stating this, we first , introduce some notation.
  	\begin{figure}[ht]
  	\begin{center}
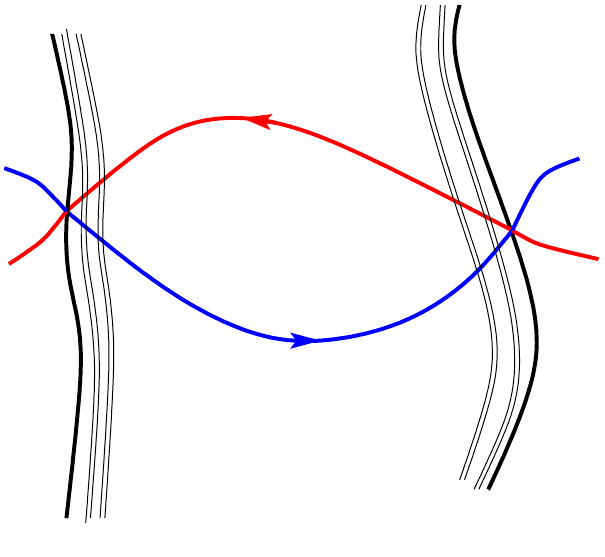
  	\end{center}
  	\caption[]{Sketch of the lamination $\mathcal{F}^c$ near the leaves $\{p_a\}$ and $\{q_a\}$, that is invariant under an iterate $\Psi$ of $\Pi$.
  		\label{f:2hom}}
  \end{figure}

	Let $\Sigma = \{0,1\}^\mathbb{Z}$ be endowed with the product topology and
	$\sigma$ be the left shift operator on $\Sigma$: 
	$(\sigma \omega)_i = \omega_{i+1}$ with $\omega = (\omega_i)_{i\in\mathbb{Z}}$. 
	We write $[\eta_0\cdots\eta_j]$ for the cylinder $\{\omega\in\Sigma \; ; \; \omega_i = \eta_i, 0\le i \le j\}$.
	Let 
	$J \subset \mathbb{R}$  be the interval $(-1,1)$ and
	$\mathcal{C}^1(\Sigma\times J)$ denote the space of skew product maps
	$F : \Sigma \times J \to \Sigma \times \mathbb{R}$, with shift dynamics in the base $\Sigma$, so that 
	\[
	F(\omega,u) = (\sigma \omega , f_{\omega} (u)),
	\] 
	with $C^1$ fiber maps $f_{\omega}$. 
	We define $I \in \mathcal{C}^1(\Sigma\times J)$ as the skew product map with identity on the fiber:
	 \[
	 I(\omega,u) \coloneqq (\sigma\omega,u).
	 \]
	We furthermore define
	\[
	F^i (\omega,u) \coloneqq (\sigma^i \omega, f^i_\omega(u)),
	\]
	where
	$f_\omega^i(u) \coloneqq f_{\sigma^{i-1}\omega} \circ \cdots \circ  f_{\sigma \omega} \circ f_\omega (u)$.
	
	Let $d_{1} (f,g)$ denote the $C^1$ distance between interval maps $f,g$ on $J$,
	\[ d_{1} (f,g) \coloneqq  \sup_{u\in J}  \left\{ |f(u)- g(u)|, |f'(u)-g'(u)|\right\}. \]
	For two skew product maps $F,G \in \mathcal{C}^1(\Sigma\times J)$
	with $F(\omega,u) = (\sigma \omega , f_\omega(u))$ and $G(\omega,u) = (\sigma\omega, g_\omega(u))$,
	we define the norm
	\[
	|F - G|_1 \coloneqq \sup_{\omega\in\Sigma} d_1 (f_\omega,g_\omega).
	\]
	We consider $\mathcal{C}^1(\Sigma\times J)$ equipped with this norm.

\begin{proposition}\label{p:skew}
		There is a sequence of decreasing neighbourhoods $\mathcal{V}_k$ of $\{ p_a\}\cup \{q_a\}$
		with the following properties:
\begin{enumerate}[(i)]		
\item		There exists a center lamination $\mathcal{F}^c$ inside $\mathcal{V}_k$ 
		with one dimensional leaves, containing the
		curves of fixed points $\mathcal{F}^c_{p_0} = \{p_a\}$ and of
		homoclinic points $\mathcal{F}^c_{q_0} = \{q_a\}$.
		The leaves $\mathcal{F}^c_v$ are continuously differentiable
		and the bundle of tangent lines $T_v\mathcal{F}^c_v$ depends continuously on $v \in \mathcal{V}_k$.	
	
\item	There is $K(k)$ with $\lim_{k\to\infty} K(k)=\infty$ and a coding $W^c: \Sigma \to \mathcal{F}^c$
such that  for   $\Psi \coloneqq \Pi^{K(k)}$   we have
	\[
	\Psi (W^c(\omega)) \cap \mathcal{V}_k \subset W^c(\sigma\omega).
	\]
	Moreover,
	$\Psi$ restricted to
	$\mathcal{F}^c$ is topologically conjugate to a skew product map
	$F_k: \Sigma \times J \to \Sigma \times \mathbb{R}$, 
	\[
	F_k (\omega , u) = (\sigma \omega , f_\omega (u)),
	\]
	where the fiber maps $f_\omega$ are continuously differentiable and depend, along with their derivatives, continuously on $\omega$.
	As $k\to \infty$, $F_k$ converges to $I$ in $\mathcal{C}^1 (\Sigma\times J)$. 
	\end{enumerate}
\end{proposition}	

In Proposition~\ref{p:skew}, the leaves of $\mathcal{F}^c$ are coded by symbol sequences in $\Sigma$.
The symbols $0$ and $1$ stand for $S_0$ and $S_1$, so that the leaf 
$W^c (\omega)$ lies in $S_{\omega_0}$
and $W^c (\sigma^k \omega) \subset S_{\omega_k}$. 
Reversibility of the system is reflected in the dynamics of the center leaves $W^c(\omega)$.
We define an involution $\mathcal{R}$ on $\Sigma$ by
\[
 \left( \mathcal{R}\omega \right)_k \coloneqq\omega_{-k}.
\] 
Then, by reversibility,
\[
R W^c (\omega) = W^c (\mathcal{R} \omega).
\]
A symbolic sequence $\omega$ is called symmetric if there exists 
$s\in\mathbb Z$ such that
\[
\mathcal{R}\omega=\sigma^s\omega.
\]

For  a periodic symbolic sequence $\omega \in \Sigma$, the leaf $W^c(\omega)$ is taken to itself by an iteration of $\Psi$, so we call it a periodic leaf.
We speak of symmetric and nonsymmetric periodic leaves.
A symbolic sequence $\omega \in \Sigma$ is called homoclinic if $\omega_i  = 0$ for $|i| \ge N$ for some $N$.
For  a homoclinic symbolic sequence $\omega \in \Sigma$, the leaf $W^c(\omega)$ is called a homoclinic leaf.
We distinguish symmetric and nonsymmetric homoclinic leaves.
We call the homoclinic leaf $W^c(\omega)$ $N$-homoclinic if 
$| \{ i \in \mathbb{Z}\; ; \; \omega_i = 1\}| = N$.

For the dynamics generated by the skew product map $F_k$ from Proposition~\ref{p:skew}, 
iterates of points may drift along fibers.
  The central issue is to identify those points whose iterates do not drift away from $\mathcal{V}_k$.
The interval diffeomorphisms $f_\omega$  that make up $F_k$ are close to the identity map, meaning there is at most a small drift in the fibers. 
Thus for each $N \in \mathbb{N}$ there is a neighbourhood $\mathcal{V}_k$
so that for each finite itinerary $\eta_0\ldots\eta_N$ and each $\omega\in [\eta_0\ldots\eta_N]$, there is an interval of points $u \in J$ for which $f^i_\omega$ is defined for $0\le i \le N$.
However statement only concerns finitely many iterates and does not help to determine the maximal invariant set 
\[
\Lambda = \bigcap_{i\in\mathbb{Z}} F_k^i(\Sigma\times J).
\] 
In particular, the intersection of $\Lambda$ with fibers $\{\omega\}\times J$
may be empty. Note that this may be the case even for points in symmetric periodic leaves of high period.

The following lemma identifies cases where large numbers of iterates stay close.

\begin{lemma}\label{l:homnodrift}
	For each $\varepsilon >0$ there is $k \in \mathbb{N}$ so that the following holds for $F_k$.
	Suppose $\omega,\chi  \in \Sigma$ and $j>0$ are so that $\omega_i=\chi_i$ for $0\le i \le j$.	
	Then 
	\[
	|f^j_{\chi} (u) - f^j_{\omega} (u) | < \varepsilon.
	\]
\end{lemma}	

\begin{proof}
	Points $v$ in the center lamination have local strong stable manifolds $\mathcal{F}^s_v$ and local strong unstable manifolds $\mathcal{F}^u_v$, see Proposition~\ref{p:strongstable}. The strong stable and strong unstable laminations are
	arbitrarily close to affine laminations if $k$ is large enough. By the assumptions,  $W^c (\sigma^i \omega)$ and $W^c(\sigma^i \chi)$ are in the same cross section $S_{\omega_i}$ for $0\le i \le j$.  Each $\mathcal{F}^s_p$, $p \in W^c(\sigma^i \omega)$, intersects a leaf 
	$\mathcal{F}^u_q$ for  a unique $q \in W^c(\sigma^i \chi)$, $0\le i \le j$. 
    As the local strong stable and local strong unstable manifolds are invariant, this connects the orbits of such points $p$ and $q$.  The result follows.

 Each $\mathcal{F}^s_p$, $p \in W^c(\omega)$, intersects a leaf 
	$\mathcal{F}^u_q$ for  a unique $q \in W^c( \chi)$.
    As the local strong stable and local strong unstable manifolds are invariant, 
this connects the orbits of such points $p$ and $q$. That is, 
$\mathcal{F}^s_{\Psi^i (p)}$ intersects 
	$\mathcal{F}^u_{\Psi^i (q)}$.
  The result follows since the strong stable and strong unstable laminations are
	close to affine for large $k$. 
\end{proof}	

\subsection{Nonsymmetric homoclinic orbits}

A key role in the argument is played by nonsymmetric homoclinic leaves in $\mathcal{F}^c$.
	Let $W^c(\omega) = \{r_a\}$ be a nonsymmetric homoclinic leaf.
We define maps $\phi_{ij}$, $0\le i,j \le 1$, by
\[
\begin{array}{ll}
\phi_{00}  : \{p_a\} \to \{p_a\}; &\; \phi_{00} (p_a) = p_a,
\\
\phi_{01}  : \{p_a\} \to \{r_a\}; &\; \phi_{01} (p_a) = W^u (p_a) \cap \{ r_a \},
\\
\phi_{10}  : \{r_a\} \to \{p_a\}; &\; \phi_{10} (r_a) = W^u (r_a) \cap \{ p_a \},
\\
\phi_{11}  : \{r_a\} \to \{r_a\}; &\; \phi_{11} (r_a) = \phi_{01}\circ\phi_{10} (r_a),
\end{array}
\]
see Figure~\ref{f:3hom}.
\begin{figure}[ht]
	\begin{center}
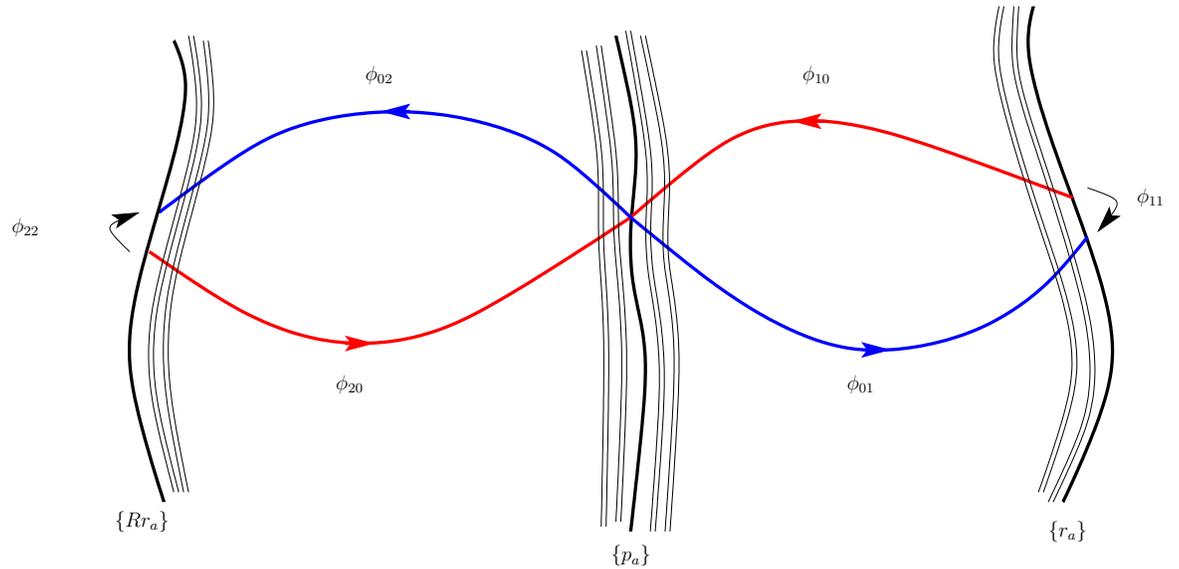
	\end{center}
	\caption[]{Sketch of nonsymmetric homoclinic leaves $\{r_a\}$ and $\{Rr_a\}$ and relevant maps between them. 
		\label{f:3hom}}
\end{figure}
Note that 
\[
\phi_{11} (r_a) = W^u ( W^u(r_a) \cap \{p_a\}) \cap \{ r_a\}.
\]

The definition of the maps $\phi_{ij}$, $0\le i ,j \le 1$ for the homoclinic leaf $\{r_a\}$ 
can be extended to the symmetric image $\{Rr_a\}$. 
We introduce a symbol $2$ and have  
maps $\phi_{ij}$, $i,j \in \{0, 2\}$ obtained from $\phi_{ij}$, $0\le i ,j \le 1$
by composing with $R$:
\[
\begin{array}{ll}
\phi_{02}: \{p_a\} \to \{R r_a\}; &\;  \phi_{02} =   R \phi_{01},
\\
\phi_{20}: \{R r_a\} \to \{p_a\}; &\;  \phi_{20} = \phi_{10} \circ R,
\\
\phi_{22}: \{R r_a\} \to \{Rr_a\}; &\;   \phi_{22}  =  R \phi_{11} \circ R.
\end{array}
\]
We  finally let $\phi_{12} = \phi_{02}\circ \phi_{10}$ and 
$\phi_{21}  =  \phi_{01}\circ \phi_{20}$.
By reversibility we have $\phi_{02} = R \phi_{10}^{-1}\circ R$ and $\phi_{20} = R\phi_{01}^{-1}\circ R$.
Upon identifying  $u \in \{r_a\}$ with $Ru \in \{Rr_a\}$ we have
$\phi_{11} = \phi_{22}^{-1}$.
Further,
$\phi_{01} = \phi_{20}^{-1}$, $ \phi_{10}=\phi_{02}^{-1}$ and
$\textrm{id}\, = \phi_{00} = \phi_{12} = \phi_{21}$.

One may consider dynamics restricted to small neighbourhoods of $\{p_a\}$ and the nonsymmetric homoclinic leaves
$\{r_a\}$ and $\{Rr_a\}$ and obtain results analogous to Proposition~\ref{p:skew} for dynamics near $\{p_a\}$ and 
the symmetric homoclinic leaf $\{q_a\}$.

Let $\Omega = \{0,1,2\}^\mathbb{Z}$ be endowed with the product topology and define the skew product system
$L: \Omega \times J \to \Omega \times \mathbb{R}$ by
\begin{align}\label{e:L}
L (\nu , u) &\coloneqq (\sigma \nu , \phi_\nu (u)),
\end{align}
where we write $\phi_\nu$ for $\phi_{\nu_0\nu_1}$.
Recall that $\mathcal{C}^1(\Omega\times J)$ denotes the space of skew product systems
on $\Omega\times J$ over the shift as base dynamics, with $C^1$ fiber maps and  
endowed with the norm
\[
|F - G|_1 \coloneqq  \sup_{\nu\in\Omega} d_1 (f_\nu,g_\nu).
\]	

\begin{proposition} \label{p:Wn}
	There are decreasing neighbourhoods $\mathcal{W}_\ell$ of $\{p_a\}\cup \{r_a\} \cup \{R r_a\}$
	with the following properties:
\begin{enumerate}[(i)]	
	\item There exists a lamination $\mathcal{G}^c \subset \mathcal{F}^c$ inside $\mathcal{W}_\ell$ and
	an iterate $\Phi = \Pi^{L(\ell)}$ so that  $\Phi$ restricted to
	$\mathcal{G}^c$ is topologically conjugate to a skew product map
	$G_\ell: \Omega \times J \to \Omega \times \mathbb{R}$,
	\[
	G_\ell (\nu , u) = (\sigma \nu , g_\nu (u)).
	\]
	\item The maps $g_\nu$ are continuously differentiable and, along with their derivatives, depend continuously on $\nu$.
	As $\ell\to \infty$, $G_\ell$ converges to $L$ in $\mathcal{C}^1(\Omega\times J)$. 
	\end{enumerate}
\end{proposition}

The symbolic code is derived from $\nu_i = 0,1,2$ if  $\Phi^i (x)$ is near $\{p_a\}$, $\{r_a\}$ or $\{R r_a\}$ respectively.
The proof of Proposition~\ref{p:Wn} is left to the reader as it only concerns a slight variation of the arguments in the proof of Proposition~\ref{p:skew}, that is presented in Section~\ref{s:clsp}.

\subsection{Scattering maps}

The following definition is inspired by \cite{delllasea00,delllasea08}. 
Assume the above setup with $W^c (\omega) = \{r_a\}$ a nonsymmetric homoclinic leaf and
$\phi_{11} (r_a) = W^u ( W^u(r_a) \cap \{p_a\}) \cap \{ r_a\}$, cf.~also Figure~\ref{f:3hom}.
\begin{definition}
	The map $\psi_\omega = \phi_{11} : \{r_a\} \to \{r_a\}$ is the scattering map for $\{r_a\}$.
\end{definition}

Take $k$ large and consider $\Psi = \Pi^{K(k)}$ on $\Sigma \times J$ as in Proposition~\ref{p:skew}.
Let $0<t<1$ and consider the following assumption:\\
\begin{description}
\item[{\it Assumption}] There exists $\zeta \in \Sigma$ with $|f^h_\zeta (0) | > t$ for some $h \in \mathbb{N}$.\\
\end{description}	

Under this assumption we find a nonsymmetric homoclinic leaf with nonidentity scattering map (Lemma~\ref{l:27})
and use Proposition~\ref{p:Wn} to study nearby dynamics.

 Take $\varepsilon$ much smaller than $t$.  
 	Given $\varepsilon$ we can take $k$ large enough as in Lemma~\ref{l:homnodrift}.
 
	\begin{lemma}\label{l:27}
		Given $\delta>0$, when choosing $k$ large enough 
		there is a homoclinic leaf $\{r_a\}$ 
		whose scattering map  $\psi$ satisfies $\psi \ne \textrm{id}$ and 
		$d_1(\psi , \textrm{id}) < \delta$.
	\end{lemma}	

\begin{proof}
	Let $\omega \in \Sigma$ be given by $\omega_i = \zeta_i$ for $0\le i \le h$ and $\omega_i = 0$ otherwise. 
	So $\omega$ is a homoclinic symbolic sequence.
	By 	Lemma~\ref{l:homnodrift} $f^h_\omega (0)$ is close to  $f^h_\zeta (0)$.
	 By the above assumption, it follows that the scattering map 
	$\psi_\omega$ is not the identity map. Necessarily $\omega$ is a nonsymmetric symbolic sequence.
	
	Knowing that there are homoclinic symbolic sequences with nonidentity scattering map, we can take $m$ minimal
	so that $\omega$ is a  homoclinic itinerary with nonidentity scattering map, $\omega_0=1$, $\omega_m=1$, and
	$\omega_i = 0$ for $i<0$ and $i>m$. This yields a scattering map that is close to the identity map, as follows.
	As $\omega$ is nonsymmetric, there exists a largest integer $r$,  $0<r<m$, with $\omega_r\ne 0$. Let $\eta\in\Sigma$ be obtained by $\eta_i = \omega_i$ for $i \ne m$ and $\eta_m = 0$.
	By the choice of $m$, the scattering map $\psi_\eta$ of the homoclinic leaf $W^c(\eta)$ is the identity map. 
	From $\eta_j = 0$ for $j > r$ we find $f^i_{\sigma^{r+1} \eta} = \text{id}$ for all $i \ge 0$.
	Compare  $f_{\sigma^{r+1} \omega}^i$ and $f_{\sigma^{r+1}\eta}^i$ for $i \ge 0$. As $\omega_j$ for $j>r$ is the symbol $0$ except for $\omega_m$ which is $1$, we find that  $f_{\sigma^{r+1} \omega}^i$
	is $C^1$-close to the identity for $k$ large. This follows by combining Proposition~\ref{p:skew} (establishing that fiber maps $f_\omega$ converge to the identity map if $k \to \infty$)
	 and Lemma~\ref{l:n} (asserting that compositions of fiber maps $f^j_{0\cdots 0}$ converge to the identity map if $k \to \infty$).
	We conclude that the scattering maps of $W^c(\eta)$ and $W^c(\omega)$ are arbitrary close for sufficiently large $k$.  
\end{proof}	
	
	\subsection{Iterated function systems}

	Consider the homoclinic leaf $\{r_a\}$ from Lemma~\ref{l:27}. Let $L$ be the corresponding skew product system
	as in \eqref{e:L}.
	Let $\sigma>0$.
	Proposition~\ref{p:Wn}  assures the existence of $\ell>0$ and a neighbourhood $\mathcal{W}_\ell$ of 
	$\{p_a\}\cup \{r_a\}\cup \{Rr_a\}$
	so that an iterate
	$\Pi^{L(\ell)}$ is topologically conjugate to a skew product map $G_\ell: \Omega\times J \to \Omega \times J$  with $|G_\ell - L|_1 < \sigma$. The fiber maps that make up $L$ are the scattering map $\psi$, its inverse $\psi^{-1}$,
	and the identity map.   We first consider the iterated function system on $J$ generated by the maps $\psi$, $\psi^{-1}$. 
		
	\begin{lemma}\label{l:ifs}
		Assume $d_1(\psi,\textrm{id})$ as in Lemma~\ref{l:27} is small. Then,
		possibly after interchanging the maps $\psi$ and $\psi^{-1}$,
		there are intervals $(a,b), (b,c), (c,d)$ inside $J$ with $\psi,\psi^{-1}$ defined on an interval containing $[a,d]$ and 
		\begin{enumerate}
			\item 
			\label{i:ifs1} $\psi^2$ maps $[a,b]$ into $(b,c)$;
			\item 
			$\psi^{-1}$ maps $[b,c]$ into $(a,c)$;
			\item 
			$\psi^{-2}$ maps $[c,d]$ into $(b,c)$;
			\item
			\label{i:ifs4} $\psi$ maps $[b,c]$ into $(b,d)$.
		\end{enumerate}
	\end{lemma}
	
	\begin{proof}
		If $\psi(u) > u$ on $J$ we can iterate a point in $(-1,-\frac{1}{2})\subset J$ many times 
		before it enters $(\frac{1}{2},1) \subset J$. We can thus take $a<b<c<d$ as
		in the statement of the lemma, as illustrated in Figure~\ref{f:ifs}.
		If $\psi(u) < u$ on $J$ we exchange $\psi,\psi^{-1}$ to get the result.
			\begin{figure}[ht]
			\begin{center}
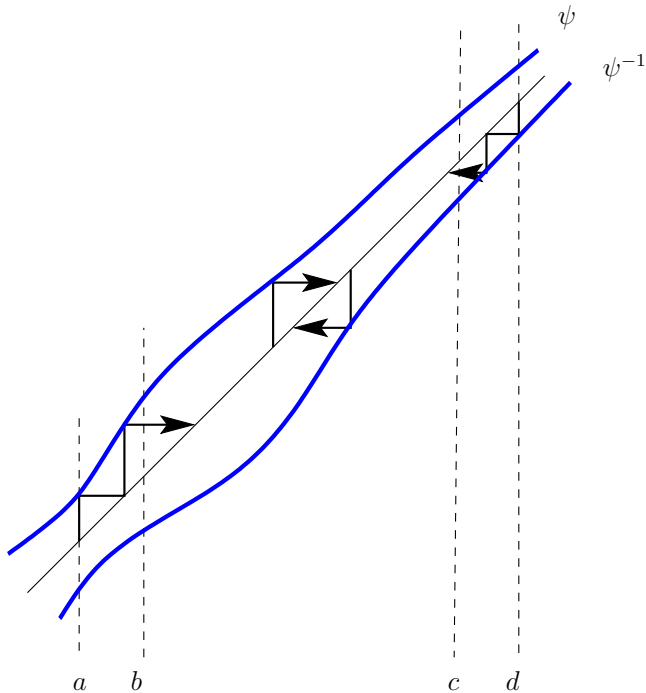
			\end{center}
			\caption[]{The scattering maps $\psi,\psi^{-1}$ define an iterated function system on an interval.
				\label{f:ifs}}
		\end{figure}	
		
		If the graph of $\psi$ does intersect the diagonal, 
		there is a point $\psi(s)=s$ and an interval $(s,v)$ on which
		$\psi (u) > u$ (possibly interchanging $\psi,\psi^{-1}$). Inside $(s,v)$ one finds an interval $(a,d)$ as above. See Figure~\ref{f:ifs2}.
		\begin{figure}[ht]
			\begin{center}
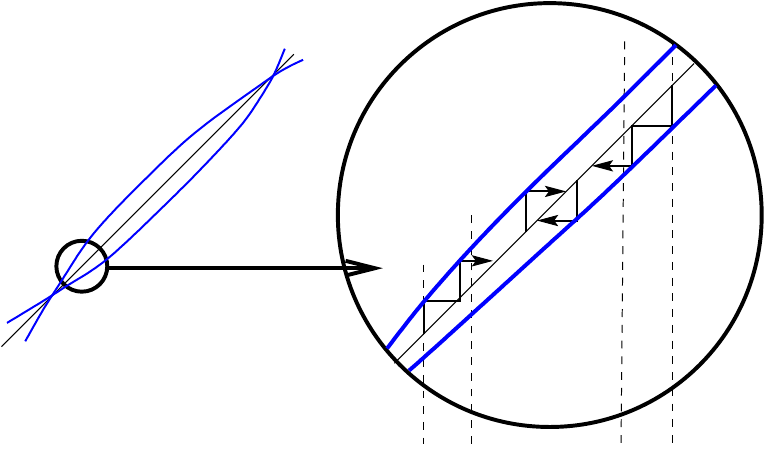
			\end{center}
			\caption[]{The scattering maps $\psi,\psi^{-1}$ define an iterated function system on an interval, here in an example where it is necessary to restrict to a subinterval on which $\psi (u) \ne u$.
				\label{f:ifs2}}
		\end{figure}	
	\end{proof}

The iterated function system generated by $\psi$, $\psi^{-1}$ leads to the skew product system 
\[
O (\omega , u) =   \left\{   \begin{array}{ll} (\sigma \omega , \psi(u)), & \omega_0 = 0,  
\\ (\sigma \omega , \psi^{-1}(u)), & \omega_0 = 1, 
\end{array}
\right.
\]
on $\Sigma \times  J$.
Assuming the conclusion of Lemma~\ref{l:ifs}, we claim that 
this skew product system on its maximal invariant set has positive topological entropy.

Projecting the maximal invariant set $\Lambda$ of $O$ to $\Sigma$ by the natural coordinate projection 
we obtain a closed invariant set for the left shift operator
on $\Sigma$. It suffices to demonstrate the latter's positive topological entropy, since it is a factor.
By \cite[Chapter~4 and Section~6.3]{linmar95} this topological entropy $h$ is given by
\[
h = \lim_{i\to \infty}  \frac{1}{i} \ln |B_i|
\]
holds, where $B_i$ is the set of blocks of length $i$ and $|B_i|$ is its cardinality.
Outside $[b,c]$  applying either $\psi^2$ or $\psi^{-2}$ maps back into $[b,c]$.
On $[b,c]$ both $\psi$ and $\psi^{-1}$ are well defined and map into $[a,d]$: on $[b,c]$ both maps can be applied.
Starting at some point in $[b,c]$, every three iterates there are at least two possible compositions available that map back into $[b,c]$. We therefore obtain $h>0$.
		
\subsection{Conclusion}
	
Take $k$ large and consider $\Psi = \Pi^{K(k)}$ on $\Sigma \times J$ as in Proposition~\ref{p:skew}.
	Let $0 < t <1$. We distinguish two possible cases, where the second case is the earlier used assumption. 
	\begin{enumerate}
		\item For each $\omega \in \Sigma$, $|f^i_\omega (0) | 	\le  t$ for all $i \in \mathbb{N}$;
		
		\item there is $\zeta \in \Sigma$ with $|f^h_\zeta (0) | > t$ for some $h \in \mathbb{N}$.\\ 
	\end{enumerate}	
	
	\noindent Case (1): The condition implies that for every $\omega\in \Sigma$ there is a point
	in $W^c (\omega)$ with itinerary $\omega$. There is thus an invariant set for $\Psi$ that projects 
	to $\Sigma$. The topological entropy of $\Psi$, and hence $\Pi$, is therefore positive. \\
	
	\noindent Case (2): Let $\{r_a\}$ be a nonsymmetric homoclinic orbit provided by Lemma~\ref{l:27}, with $d_1(\psi,\textrm{id}) < \delta$ small.
	Consider the skew product system $G_\ell: \Omega \times J \to \Omega \times \mathbb{R}$ provided by Proposition~\ref{p:Wn}. Let $\ell$ be large
	and $\Omega' \subset \Omega$ be the space of symbolic sequences in $\Omega$ in which $0$ does not occur and $121$ and $212$ also do not occur. The shift on $\Omega'$ is a subshift of finite type. Restricting $G_\ell$ to $\Omega'\times J$,
	the skew product map $(\nu,u) \mapsto G_\ell (\nu,u)$ has fiber maps $g_\nu$  which, for $\ell$ large enough, are arbitrarily close to
	$\textrm{id}$, $\psi$, or $\psi^{-1}$.  Moreover, since $121$ and $212$ are forbidden subsequences, there is at most one identity map in a row. If we compose any fiber map close to the identity map with the subsequent map (which is either $\psi$ or $\psi^{-1}$), we obtain a composition
	of maps that are arbitrarily close to $\psi$ or $\psi^{-1}$, for large enough $\ell$.
	The argument to prove positive topological entropy for the iterated function system generated by $\psi$, $\psi^{-1}$,
	given above, applies to show that the skew product map $G_\ell$ has positive topological entropy on its maximal invariant set. Indeed, although the fiber maps are not exactly $\psi$ or $\psi^{-1}$, they are $C^1$ close so that 
	the properties \eqref{i:ifs1}-\eqref{i:ifs4} listed in Lemma~\ref{l:ifs} still hold for $\psi$, $\psi^{-1}$ replaced by the nearby fiber maps. 
	
	We proceed  to call a cylinder $\Omega = [\omega_{-m}\cdots\omega_n]$ {\em good} if for any infinite $\omega \in \Omega$ we have $g_\omega^i (0) \in (a,d)$ for $-m \le i \le n$. Write for instance $\Omega\, 2$ for the cylinder
	$[\omega_{-m}\cdots\omega_n 2]$.
	By construction, if $\Omega$ is good, then both $\Omega\, 1$ and $\Omega\, 2$ are also good if $g_\omega^n (0) \in (b,c)$. Otherwise, if $g_\omega^n (0) > c$, then $\tilde{\Omega} = \Omega\, 22$ is good, and if $g_\omega^n (0) < b$, then $\tilde{\Omega} = \Omega\, 11$ is good, and
	$g_{\tilde\omega}^{n+2} (0) \in (b,c)$ in both these cases, so both $\tilde{\Omega}\, 1$ and $\tilde{\Omega}\, 2$ are good. The same is true for the expansion of $\Omega$ to the left. Thus, one constructs the growing set of cylinders such that in the limit we obtain the set of codes corresponding to fibers which contain a bounded orbit. As the number of good cylinders grows exponentially with length, the topological entropy is positive.
	 This concludes the proof of Theorem~\ref{t:mainpi}.

\section{Local normal forms}\label{s:lnf}

The final two sections of this paper develop some important technical results, 
in particular expansions  for the local transition map and existence and regularity of invariant center laminations.
We first develop expansions of the return map on $S_0$ close to $p_0$. 
Let $\mathcal{U}_0$ be a small tubular neighbourhood of $\gamma_0$.
Consider the first return map $\Pi_0: S_0\to S_0$ considering orbits in $\mathcal{U}_0$. Note that $\Pi_0$ is defined on an open neighbourhood of $p_0$ in $S_0$, nonetheless we write $\Pi_0: S_0\to S_0$.
Local stable and unstable manifolds of $p_a$ are denoted by 
$W^s_{\textrm{loc}} (p_a)$ and $W^u_{\textrm{loc}} (p_a)$; they are intersection with $S_0$ of local stable and unstable manifolds
$V^s_{\textrm{loc}} (\gamma_a)$ and $V^u_{\textrm{loc}} (\gamma_a)$ of the periodic orbit $\gamma_a$ for the flow.

\begin{lemma}\label{l:normalform}
	There exist smooth coordinates $(x,y,z)$ from an open set in $\mathbb{R}^{n-1}\times \mathbb{R}\times \mathbb{R}^{n-1}$ on $S_0$ such that
	$p_a = (0,0,a)$ and
	$(x_1,y_1,z_1) = \Pi_0 (x_0,y_0,z_0)$ has the following expression:
	\begin{align}
	\nonumber
	x_{1} &=  A(z_0) x_0 + e(x_0,y_0,z_0),
	\\
	\nonumber
	y_{1} &= A^{-1}(z_0) y_0 + f(x_0,y_0,z_0),
	\\
	\label{xyz}
	z_{1} &=  z_0 + g(x_0,y_,z_0), 
	\end{align}
	where the matrix $A \in \mathbb{R}^{n-1 \times n-1}$ depends smoothly on $z_0$
	and satisfies
	$ \| A (0) x \| < \lambda \| x\|$ for some $0< \lambda <1$.
	Further, $e,f,g$ are smooth functions satisfying
	\begin{align}
	\nonumber
	e(x_0,y_0,z_0) &=  \mathcal{O}(x_0^2) + 	\mathcal{O} (x_0 y_0),
	\\
	\nonumber
	f(x_0,y_0,z_0) &= 	\mathcal{O} (y_0^2) + \mathcal{O} (x_0 y_0),
	\\
	\label{efg}
	g(x_0,y_0,z_0) &= \mathcal{O} (x_0 y_0).	
	\end{align}
\end{lemma} 

\begin{proof}
	Take local coordinates $(x,y,z)$ near $p_0$ for which 
	$p_a = (0,0,a)$ and 
	\begin{align*}
	W^u_{\textrm{loc}} (p_a) &=  \{x=0,z = a\}, 
	\\
	W^s_{\textrm{loc}} (p_a) &=  \{y=0,z = a\}.
	\end{align*}
	Note that $W^u_{\textrm{loc}}(p_a)$ and $W^s_{\textrm{loc}}(p_a)$ are invariant for every $a$, from which expressions (\ref{xyz}) and estimates (\ref{efg}) follow.  
	By the reversibility, the restriction of the map on $W^s_{\textrm{loc}} (p_a)$ is conjugate to the restriction of the inverse map on $W^u_{\textrm{loc}} (p_a)$,
	therefore the linearization matrices, at $p_a$, of the map on $W^s_{\textrm{loc}} (p_a)$ and the map on $W^u_{\textrm{loc}} (p_a)$ are, indeed, 
	inverse to each other.
\end{proof}

The curve $\{p_a\}$ of fixed points has a local center stable manifold 
$W^{sc}_{loc} (\{p_a\})$, which is a union of local  stable manifolds 
$W^s_{loc} (p_a)$ of individual fixed points.
Similarly the local unstable manifold 
$W^{cu}_{loc} (\{p_a\})$  is foliated by local unstable manifolds
$W^u_{loc}(p_a)$.

We may take the coordinates in Lemma~\ref{l:normalform} such that the action of $R$ is locally linear \cite{bochner} and given by $R (x,y,z) = (y,x,z)$. 
Indeed, $R W^u_{\textrm{loc}} (p_a) = W^s_{\textrm{loc}} (p_a)$. Therefore, writing $R(x,y,z)=(R_1(x,y,z), R_2(x,y,z), R_3(x,y,z))$, we have
$$R_1(x,0,z)=0, \qquad R_2(0,y,z) = 0, \qquad R_3(0,x,z)=R_3(x,0,z)=z.$$
It also follows that the derivative of $R$ at zero is $R'(0): (x,y,z) \mapsto (y,x,z)$. These formulas imply that the coordinate transformation
$$(x,y,z)^{new} = \frac{1}{2} [(x,y,z) + R'(0) R (x,y,z)]$$
linearizes the involution $R$ while maintaining the validity of the expressions (\ref{xyz}) and estimates (\ref{efg}).

Fix $\lambda$ with $\| A(0)\| < \lambda < 1$ and take $\varepsilon>0$ so that for
$|z| < \varepsilon$,  
\[
\|A(z)\| < \lambda.
\]
We further assume that the neighbourhood $\mathcal{U}_0$ of the periodic orbit under consideration is chosen small enough so that $\mathcal{U}_0 \cap S_0$ 
is contained in $\{|z| < \varepsilon\}$. 

The following lemma yields estimates for iterates near $p_0$ in cross coordinates.
See \cite{shishiturchu98,shishiturchu01} for an introduction to the theory and use of cross coordinates in dynamics.
Note that the estimates improve for orbits with higher number of iterates that stay closer to the local stable and unstable manifolds of $p_a$. 

\begin{lemma}\label{l:n}
	Assume $(x_i,y_i,z_i) = \Pi_0^i (x_0,y_0,z_0) \in S_0$ for all $0\le i \le k$.
	Then $(x_k,y_k,z_k) = \Pi_0^k (x_0,y_0,z_0)$ can be solved for
	$(x_k,y_0,z_k)$ as function of $(x_0,y_k,z_0)$, with the following expression:
	\begin{align*}
	x_k &= 
	R_x (x_0,y_k,z_0),
	\\
	y_0 &=  
	R_y(x_0,y_k,z_0),
	\\
	z_k &= z_0 + R_z(x_0,y_k,z_0).
	\end{align*}
	For some $C>0$, independent of $k$, 
	\begin{align*}
	|R_x (x_0,y_k,z_0)|, |DR_x (x_0,y_k,z_0)| &\le C  \lambda^{k},
	\\
	|R_y (x_0,y_k,z_0)|, |DR_y (x_0,y_k,z_0)| &\le C  \lambda^{k},
	\\
	|R_z (x_0,y_k,z_0)|, |DR_z (x_0,y_k,z_0)| &\le C  \lambda^{k/2}.
	\end{align*}		
\end{lemma}

\begin{proof}
This follows from \cite[Lemma~1]{geltur17}.	
\end{proof}

As $q_a \in W^{sc} (\{p_a\})$ we can also define the
local center stable manifold 
$W^{sc}_{loc} (\{q_a\})$ of the curve $\{q_a\}$ of homoclinic points
as the connected component of  $W^{sc} (\{p_a\}) \cap S_1$ that contains $\{q_a\}$.
It is foliated by local stable manifolds $W^s_{loc} (q_a)$.
Similarly, the
local center unstable manifold 
$W^{sc}_{loc} (\{q_a\})$
is foliated by local unstable manifolds $W^u_{loc} (q_a)$.
Observe that
\begin{align*}
\{p_a\} &=  W^{sc}_{loc} (\{p_a\}) \cap W^{cu}_{loc} (\{p_a\}),
\\
\{q_a\} &=  W^{sc}_{loc} (\{q_a\}) \cap W^{cu}_{loc} (\{q_a\}).
\end{align*}
Throughout, we assume smooth coordinates $(x,y,z)$ on $S_1$ in which
\begin{align*}
W^{sc}_{loc} (\{q_a\}) &=  \{y=0\},
\\
W^{cu}_{loc} (\{q_a\}) &=  \{x=0\},
\end{align*}
and also
\begin{align}
\nonumber
W^s_{loc} (q_a) &=  \{y=0,z=a\},
\\
\label{e:coordnearq}
W^u_{loc} (q_a) &= \{x=0,z=a\}.
\end{align}

\section{Center laminations}\label{s:clsp}

The existence of a locally invariant center lamination $\mathcal{F}^c$ of one dimensional leaves is key to Proposition~\ref{p:skew}.
This center lamination extends the curves of fixed points $\{p_a\}$ and homoclinic points $\{q_a\}$. Its
existence and stated regularity properties, together with the normal form expansions
in cross coordinates obtained in the previous section, establish Proposition~\ref{p:skew}.
For discussions on the existence of center laminations in similar contexts, see also \cite{homkno06,homlam06}. 

Our proof of the next proposition is related to the cross coordinates from Lemma~\ref{l:normalform} and Lemma~\ref{l:n}. 
We let $\mathcal{F}^c_u$ represent the leaf of $\mathcal{F}^c$ through a point $u$.

\begin{proposition}\label{p:Vn}
	There is a sequence of decreasing neighbourhoods $\mathcal{V}_k$ of $\{ p_a\}\cup \{q_a\}$
	with iterates $\Psi = \Pi^{K(k)}$ so that: 
	\begin{enumerate}[(i)] 
	\item 	
	There exists a center lamination $\mathcal{F}^c$ inside $\mathcal{V}_k$ 
	with one dimensional leaves, containing the
	curves of fixed points $\mathcal{F}^c_{p_0} = \{p_a\}$ and of
	homoclinic points $\mathcal{F}^c_{q_0} = \{q_a\}$.
	The leaves $\mathcal{F}^c_v$ are continuously differentiable
	and the bundle of tangent lines $T_v\mathcal{F}^c_v$ depends continuously on $v\in\mathcal{V}_k$.
	
	\item There is a homeomorphism $W^c: \Sigma \to \mathcal{F}^c$ with \[\Psi (W^c(\omega)) \subset W^c(\sigma\omega).\]
	\end{enumerate}
\end{proposition}

\begin{proof}
	Consider local coordinates $(x,y,z) \in \mathbb{R}^{n-1}\times\mathbb{R}\times\mathbb{R}^{n-1}$ near $p_0$ as in Lemma~\ref{l:normalform}.
	We can find  arbitrary large integers $k$ and sets $V_1$, $V_2$ that are small neighbourhoods of 
	$W^{sc}_{loc} (p_0)$, $W^{cu}_{loc}(p_0)$,
	with 
	\[
	V_2 = \Pi^k(V_1)
	\] 
	and 
	\begin{align*}
	V_1 &\subset \{ |x| \le 1, |y| \le C\lambda^k , |z| \le c  \}, 
	\\
	V_2 & \subset 
	\{ |x| \le C\lambda^k , |y| \le 1, |z| \le c\}
	\end{align*}
	for uniform constants $C>0$, $c>0$.
	This is similar to the  construction for (planar) diffeomorphisms in
	\cite{paltak93}, with an additional $z$-direction included, see Figure~\ref{f:horseshoe}. 
	\begin{figure}[ht]
		\begin{center}
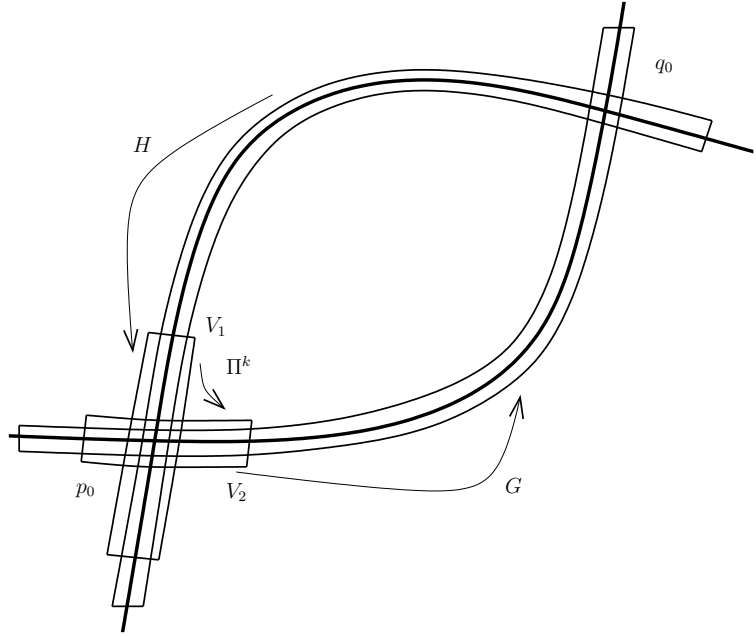
		\end{center}
		\caption[]{Sketch of the main elements in the 
		construction of a horseshoe near a homoclinic tangle as employed in the proof of Proposition~\ref{p:Vn}, omitting
		the additional center coordinate and under the assumption of a single cross section.
			\label{f:horseshoe}}
	\end{figure}
	
	There exists $l_+>0$ so that $q_0 \in \Pi^{l_+} (V_2)$ and
	$l_->0$ so that $\Pi^{l_-} (q_0) \in V_1$. 
	The numbers $l_+,l_-$ do not depend on $k$. (By taking $S_0$ and $S_1$ small we may in fact assume $l_- = l_+ = 1$.)
	Note that $\Pi^{l_- + k + l_+}$ maps $\Pi^{-l_-}(V_1)$ to $\Pi^{l_+}(V_2)$.
	Adjusting the sets $V_1,V_2$ we may assume that $q_0$ is in the interior of 
	$\Pi^{l_+} (V_2)$ and of $\Pi^{-l_-} (V_1)$.
	For $k$ large  we find that
	$\Pi^{l_-+k+l_+}$ maps $\Pi^{-l_-}(V_1)$ to a set that 
	has an intersection with $\Pi^{-l_-}(V_1)$ near $p_0$ and near $q_0$.
	Write $Q_i = \Pi^i (q_0)$.

	The map $\Psi = \Pi^{l_- + k + l_+}$ on $\Pi^{-l_-} (V_1)$ consists of maps 
	$\Psi_{00}: S_0  \to S_0$,
	$\Psi_{01}: S_0  \to S_1$,
	$\Psi_{10}: S_1  \to S_0$
	and
	$\Psi_{11}: S_1 \to S_1$ (the domains are open sets contained in the indicated cross sections).
	The maps have the form	$\Psi_{00} = \Pi^{l_- + k + l_+}$,
	$\Psi_{01} = G \circ \Pi^{l_- + k}$, 
	$\Psi_{10} = \Pi^{k + l_+} \circ H$,
	$\Psi_{11} = G \circ \Pi^k \circ H$.
	Here $G = \Pi^{l_+}$ is a local diffeomorphism that maps a neighbourhood of
	$Q_{-l_+}$ to a neighbourhood of $q_0$ and
	$H = \Pi^{l_-}$ is a local diffeomorphism that maps a neighbourhood of
	$q_0$ to a neighbourhood of $Q_{l_-}$.
	
	Near $q_0$ in $S_1$ we take coordinates $(x,y,z)$ satisfying \eqref{e:coordnearq}.
	We write all four maps $\Psi_{ij}$, $0\le i,j\le 1$, in cross coordinates, as in Lemma~\ref{l:n}.
	This defines maps
	$\Phi^{+}_{ij} : \mathbb{R}^n \to \mathbb{R}^n$
	with 
	\[
	\Phi^{+}_{ij} (\hat{x}_0,\hat{y}_1,\hat{z}_0) = (\hat{x}_1,\hat{y}_0,\hat{z}_1)
	\] 
	if 
	\[
	\Psi_{ij} (\hat{x}_0,\hat{y}_0,\hat{z}_0)  =(\hat{x}_1,\hat{y}_1,\hat{z}_1).
	\]

	and $(\hat{x}_1,\hat{y}_1,\hat{z}_1) = (x_K,y_K,z_K)$  for $K = l_-+ k + l_+$.
	Recall that by Lemma~\ref{l:n} we have for $i,j = 0,0$, 
	\begin{align*}
	\Phi^{+}_{00}  (\hat{x}_0,\hat{y}_1,\hat{z}_0) &= 
	\left( \begin{array}{c}
	S_{00,x} (\hat{x}_0,\hat{y}_1,\hat{z}_0) \\ S_{00,y}(\hat{x}_0,\hat{y}_1,\hat{z}_0) \\ \hat{z}_0 + S_{00,z}(\hat{x}_0,\hat{y}_1,\hat{z}_0)
	\end{array}\right).
	\end{align*}
	Bounds for the terms on the right-hand side are given in Lemma~\ref{l:n}.
	
	Similar estimates hold for other pairs $i,j$. Consider for instance for $\Phi^{+}_{01}$.
	Note that
	$G$ maps $W^{u}_{loc} (p_a)$ to $W^{u}_{loc} (q_a)$, thus
	$G$ maps $\{x=0,z=a\}$ to $\{x=0,z=a\}$. 
	As further $G(Q_{-l_+}) = (0,0,0)$, we have
	\[
	G((x,y,z) - Q_{-l_+}) = \left( \begin{array}{c} 
	a_{11} x + \mathcal{O}(x)\mathcal{O}(\|(x,y,z\|) 
	\\ 
	a_{21} x + a_{22} y + a_{23}z +\mathcal{O} (\|x,y,z\|^2) 
	\\
	z +  \mathcal{O} (x)  \mathcal{O}(\|(x,y,z\|)
	\end{array}
	\right) 
	\]
	By the implicit function theorem,  
	\begin{align*}
	\Phi^{+}_{01}  (\hat{x}_0,\hat{y}_1,\hat{z}_0)  &= 
	\left( \begin{array}{c}
	S_{01,x} (\hat{x}_0,\hat{y}_1,\hat{z}_0) \\ S_{01,y}(\hat{x}_0,\hat{y}_1,\hat{z}_0) \\ \hat{z}_0 + S_{01,z}(\hat{x}_0,\hat{y}_1,\hat{z}_0)
	\end{array}\right)
	\end{align*}
	with for some $C>0$ independent of $k$, 
	\begin{align*}
	|S_{01,x} (\hat{x}_0,\hat{y}_1,\hat{z}_0)|, |DS_{01,x} (\hat{x}_0,\hat{y}_1,\hat{z}_0)| &\le C  \lambda^{k},
	\\
	|S_{01,y} (\hat{x}_0,\hat{y}_1,\hat{z}_0)|, |DS_{01,y} (\hat{x}_0,\hat{y}_1,\hat{z}_0)| &\le C  \lambda^{k},
	\\
	|S_{01,z} (\hat{x}_0,\hat{y}_1,\hat{z}_0)|, |DS_{01,z} (\hat{x}_0,\hat{y}_1,\hat{z}_0)| &\le C  \lambda^{k/2}.
	\end{align*}
	Such estimates exist also for $\Phi^{+}_{10}$ and $\Phi^{+}_{11}$.
	
	Let $\theta : \mathbb{R}^{2n-1} \to \mathbb{R}$ 
	be a nonnegative test function,
	with $\theta \equiv 1$ near the origin and $\theta \equiv 0$ outside a
	neighbourhood of the origin.
	For $\epsilon > 0$, let $\theta_\epsilon (x) = \theta(x/\epsilon)$.
	Replace $\Pi$ on $S_0$ by $\theta_\epsilon \Pi + (1 - \theta_\epsilon) D\Pi (0,0,0)$.
	By rescaling we may assume $\Pi$ that is unaltered on  $\{ |x| \le 1, |y|\le 1, |z|\le 1\}$ and linear outside
	$\{ |x| \le 2, |y|\le 2, |z|\le 2\}$. 
	We can now consider $\Psi_{00}$ on a uniform neighbourhood $\{ |x|\le 3, |y|\le 3\} \times \mathbb{R}$ of the $z$-axis in $\mathbb{R}^{2n-1}$. 
	Likewise, we extend the local diffeomorphisms $G$ and $H$ to have all
	maps $\Psi_{ij}$ defined on a uniform neighbourhood $\{ |x|\le 3, |y|\le 3\} \times \mathbb{R}$ of the $z$-axis in $\mathbb{R}^{2n-1}$. Moreover, we may assume that  $\Psi_{ij}$ is affine and acts as the identity map on the $z$-coordinate, outside $\{ |x| \le 2, |y|\le 2, |z|\le 2\}$.
	
	As a result we find $\Phi^{+}_{ij}$, $0\le i,j \le 1$, which we will consider on 
	$\{ |x|\le 3, |y|\le 3, |z|\le 3\}$,
	with 
	\begin{align}\label{e:Phi+ij}
	\Phi^{+}_{ij}  (\hat{x}_0,\hat{y}_1,\hat{z}_0) &= 
	\left( \begin{array}{c}
	S_{ij,x} (\hat{x}_0,\hat{y}_1,\hat{z}_0) \\ S_{ij,y}(\hat{x}_0,\hat{y}_1,\hat{z}_0) \\ \hat{z}_0 + S_{ij,z}(\hat{x}_0,\hat{y}_1,\hat{z}_0)
	\end{array}\right)
	\end{align}
	and for some $C>0$ independent of $k$, 
	\begin{align}
	\nonumber
	|S_{ij,x} (\hat{x}_0,\hat{y}_1,\hat{z}_0)|, |DS_{ij,x} (\hat{x}_0,\hat{y}_1,\hat{z}_0)| &\le C  \lambda^{k} |\hat{x}_0|,
	\\
	\nonumber
	|S_{ij,y} (\hat{x}_0,\hat{y}_1,\hat{z}_0)|, |DS_{ij,y} (\hat{x}_0,\hat{y}_1,\hat{z}_0)| &\le C  \lambda^{k} |\hat{y}_1|,
	\\
	\label{e:Sij}
	|S_{ij,z} (\hat{x}_0,\hat{y}_1,\hat{z}_0)|, |DS_{ij,z} (\hat{x}_0,\hat{y}_1,\hat{z}_0)| &\le C  \lambda^{k/2}.
	\end{align}

	In the same way, we define maps	$\Phi^{-}_{ij}$, also considered on	$\{ |x|\le 3, |y|\le 3, |z|\le 3\}$,
	with 
	\[
	\Phi^{-}_{ij} (\hat{x}_0,\hat{y}_1,\hat{z}_1) = (\hat{x}_1,\hat{y}_0,\hat{z}_0)
	\] 
	(so the central $z$-coordinate is treated differently) if 
	\[
	\Psi_{ij} (\hat{x}_0,\hat{y}_0,\hat{z}_0)  =(\hat{x}_1,\hat{y}_1,\hat{z}_1).
	\]
	This can be viewed as cross  coordinates for the inverse map $\Psi^{-1}$. Asymptotic expansions for $\Phi^{-}_{ij}$ are like those for $\Phi^{+}_{ij}$:
	\begin{align}\label{e:Phi-ij}
	\Phi^{-}_{ij}  (\hat{x}_0,\hat{y}_1,\hat{z}_1) &= 
	\left( \begin{array}{c}
	T_{ij,x} (\hat{x}_0,\hat{y}_1,\hat{z}_1) 
	\\ 
	T_{ij,y}(\hat{x}_0,\hat{y}_1,\hat{z}_1) 
	\\ 
	\hat{z}_1 + T_{ij,z}(\hat{x}_0,\hat{y}_1,\hat{z}_1)
	\end{array}\right)
	\end{align}
	with
	\begin{align}
	\nonumber
	|T_{ij,x} (\hat{x}_0,\hat{y}_1,\hat{z}_1)|, |DT_{ij,x} (\hat{x}_0,\hat{y}_1,\hat{z}_1)| &\le C  \lambda^{k} |\hat{x}_0|,
	\\
	\nonumber
	|T_{ij,y} (\hat{x}_0,\hat{y}_1,\hat{z}_1)|, |DT_{ij,y} (\hat{x}_0,\hat{y}_1,\hat{z}_1)| &\le C  \lambda^{k} |\hat{y}_1|,
	\\
	\label{e:Tij}
	|T_{ij,z} (\hat{x}_0,\hat{y}_1,\hat{z}_1)|, |DT_{ij,z} (\hat{x}_0,\hat{y}_1,\hat{z}_1)| &\le C  \lambda^{k/2},
	\end{align}
	for some $C>0$ independent of $k$.

	Now fix $\omega \in \Sigma$.	
	Denote by 		$\mathcal{C}(\mathbb{Z}, \mathbb{R}^{2n-1})$ the space of bounded sequences 
	$\xi:\mathbb{Z} \to \mathbb{R}^{2n-1}$ endowed with the supnorm.
	Consider its subset $\mathcal{C}(\mathbb{Z},\{ |x|\le 3, |y|\le 3, |z|\le 3 \})$ consisting of the sequences 
	$\xi:\mathbb{Z} \to \{ |x|\le 3, |y|\le 3, |z|\le 3\}$.
	Abbreviate
	\[
	\mathcal{C} = \mathcal{C}(\mathbb{Z},\{ |x|\le 3, |y|\le 3, |z|\le 3 \}).
	\]
	For fixed $z_0$ with $|z_0|\le 3$, 
	define 
	\[
	\mathcal{H}: \mathcal{C}  \to \mathcal{C}(\mathbb{Z},\mathbb{R}^{2n-1})
	\] 
	as follows:
	if $\gamma_i = (x_i,y_i,z_i)$ and $\mathcal{H} (\gamma) = \eta$ with 
	$\eta_i = (u_i,v_i,w_i)$, then
	\begin{align}
	\nonumber
	(u_{i+1},v_i,w_{i+1}) &=  
	\Phi^{-}_{\omega_i\omega_{i+1}} (x_i,y_{i+1},z_{i+1}),  \; \textrm{if } i \ge 0,
	\\
	\label{e:mathcalH} 
	(u_{i+1},v_i,w_{i}) &=   \Phi^{+}_{\omega_i\omega_{i+1}} (x_i,y_{i+1},z_i),  \;  \textrm{if } i < 0.
	\end{align}
	for $i \in \mathbb{Z}$, and $w_0 = z_0$.

	Orbits of $\Pi$ are fixed points of $\mathcal{H}$. 
	By \eqref{e:Phi+ij}, \eqref{e:Phi-ij}, the estimates \eqref{e:Sij}, \eqref{e:Tij},  and the fact that $\mathcal{H}$ 
	acts as the identity map on the $z$-coordinate outside $\{|x|\le 2, |y|\le 2, |z|\le 2\}$,
	we find that 	$\mathcal{H}$ maps
	$\mathcal{C}$ into itself:
	\[
	\mathcal{H} (\mathcal{C})  \subset \mathcal{C}.
	\]
	The map $\mathcal{H}$ is, however, not a contraction on  $\mathcal{C}$, due to the existence of a central direction. To remedy this,
	we find a contraction by using scaled Banach spaces as in  \cite{van89}. 
	Write $\mathcal{C}_\alpha(\mathbb{Z},\mathbb{R}^{2n-1})$ for the set of sequences
	$\mathbb{Z} \to \mathbb{R}^{2n-1}$ 
	with $\sup_{i \in \mathbb{Z}} \alpha^{-|i|} \|\gamma(i)\| < \infty$ and
	equipped with the norm
	\[
	\|\gamma\|_\alpha = \sup_{i \in \mathbb{Z}} \alpha^{-|i|} \|\gamma(i)\|.
	\]
	The set $\mathcal{C}_\alpha (\mathbb{Z},\{ |x|\le 3, |y|\le 3, |z|\le 3\} )$ consist
	of those sequences $\gamma \in \mathcal{C}_\alpha(\mathbb{Z},\mathbb{R}^{2n-1})$ with $\gamma_i \in 
	\{ |x|\le 3, |y|\le 3, |z|\le 3\}$.  Abbreviate 
	\[
	\mathcal{C}_\alpha = \mathcal{C}_\alpha(\mathbb{Z},\{ |x|\le 3, |y|\le 3, |z|\le 3\}).
	\]
	
	Let $\alpha>1$ be fixed and close to $1$ and 
	consider $\mathcal{H}$ on  $\mathcal{C}_\alpha$.	
	From \eqref{e:mathcalH} we find, for $i \ge 0$, 
	\begin{align*}
	u_i &=  S_{\omega_{i-1}\omega_i} ( x_{i-1}, y_{i}, z_{i-1}),
	\\
	v_i &= S_{\omega_{i}\omega_{i+1}} ( x_i, y_{i+1}, z_i),
	\\
	w_i &= w_{i-1} + S_{\omega_{i-1}\omega_i} ( x_{i-1}, y_{i}, z_{i-1}).
	\end{align*}
	For $i<0$ there are similar  expressions of the form
	\begin{align*}
	u_i &=  T_{\omega_{i-1}\omega_i} ( x_{i-1}, y_{i}, z_{i}),
	\\
	v_i &= T_{\omega_{i}\omega_{i+1}} ( x_i, y_{i+1}, z_{i+1}),
	\\
	w_{i} &= w_{i+1} + T_{\omega_{i}\omega_{i+1}} ( x_{i}, y_{i+1}, z_{i+1}).
	\end{align*}
	One easily checks from this,  using $\alpha>1$ and \eqref{e:Phi+ij}, \eqref{e:Sij},  \eqref{e:Phi-ij}, \eqref{e:Tij},  that 
	$\mathcal{H}$ is a contraction on $\mathcal{C}_\alpha$.
	The map $\mathcal{H}$ therefore possesses a unique fixed point $\zeta$; 
	$\zeta(i)$ is an orbit for $\Psi$.
	Write  $w(z_0) = \zeta(0)$. This gives $w:  \{|z|\le 3\} \to \{ |x| \le 3, |y| \le 3\}$.
	It follows from the reasoning in \cite{van89}
	that $w$ is continuously differentiable. 	
	(Apart from the nonessential difference that the expression of the contraction is not the same,  the additional structure consists of a dependence on $\omega \in \Sigma$.  We do not have a single map but a finite number of maps $\Phi^\pm_{ij}$ coded by $\omega$ and that appear in the expressions, but the reasoning in \cite{van89} can be followed.)	
	Define
	\[
	W^{c} (\omega) = \bigcup_{z_0 \in \mathbb{R}} w(z_{0});
	\]
	$W^{c}(\omega)$ is the sought for center manifold.
	By construction, $W^c (\omega)$ (intersected  with $S$) is contained in $S_{\omega_0}$ and satisfies
	\[
	\Psi (W^c(\omega)) \subset W^c (\sigma \omega).
	\]
	
	Note that $ \alpha^{-|n|} \|\gamma_n\| \to 0$ as $|n|\to \infty$, for 
	$\gamma \in \mathcal{C}_\alpha$.
	Therefore $\mathcal{H}: \mathcal{C}_\alpha \to \mathcal{C}_\alpha$   
	depends continuously on $\omega \in \Sigma$. It follows that $w$ depends continuously on $\omega$.
	The contraction $\mathcal{H}: \mathcal{C}_\alpha \to \mathcal{C}_\alpha$ is not continuously differentiable, but it is continuously differentiable
	when considering 
	$\mathcal{H}: \mathcal{C}_{\alpha'}  \to \mathcal{C}_{\alpha}$ for $\alpha ' < \alpha$. 
	See \cite{van89}, also to see that this implies that $w$ is continuously differentiable.
	As the derivatives of $\mathcal{H}: \mathcal{C}_{\alpha'}  \to \mathcal{C}_{\alpha}$ depend continuously on $\omega$, the derivative of  $w$ depends continuously on $\omega$.
\end{proof}

\begin{remark}
 A simplified version of the above proof in which the center direction is ignored, gives an analytical proof based on cross coordinates of the existence of a horseshoe near homoclinic tangles of general systems.
 The usual proof, as in  \cite{kathas95}, deploys invariant cone fields.
 \end{remark}
 
\begin{proof}[Proof of Proposition~\ref{p:skew}]
	The proof of Proposition~\ref{p:Vn} introduces sets $V_1, V_2 = \Pi^k (V_1)$ near $p_0$.
	Together with iterates $\Pi^{j}(V_1)$, $ -l_-\le j \le -1$ and $\Pi^j (V_2)$, $1\le j \le l_+$, this defines a small neighbourhood $V$ of the closure of the orbit under $\Pi$ of $q_0$. 
	
	With $K = l_-+k+l_+$ as in 
	the proof of  Proposition~\ref{p:Vn} we observe that  $\mathcal{V}_k = V \cap \Pi^{K} (V)$ is a small neighbourhood of $p_0 \cup q_0$. 	
	The iterate $\Psi = \Pi^{K}$  possesses an invariant center lamination 
	inside $V \cap \Pi^{K} (V)$ with leaves $W^c (\omega)$, $\omega \in\Sigma$.

	Consider the fiber map from $W^c (\omega)$ to $W^c (\sigma \omega)$. To fix thoughts, assume $\omega_0 = \omega_1 = 0$,
	so that the fiber map is obtained by iterating the local map $\Pi$ on $S_0$. The other possible cases are treated similarly.
	Lemma~\ref{l:n} provides expressions for points $u = (x_0,y_0,z_0) \in W^c (\omega)$ and $\Pi^{K} (u) = (x_K,y_K,z_K) \in W^c (\sigma \omega)$. 
	Now $u \in  W^c (\omega)$ gives $x_0,y_0$ as function of $z_0$ and $\Pi^{K} (u) \in W^c (\sigma \omega)$ gives
	$x_K,y_K$ as function of $z_K$.
	By the implicit function theorem we solve  
	\begin{align*}z_K &= z_0 + R_z (x_0(z_0),y_K (z_K),z_0)
	\end{align*} from Lemma~\ref{l:n} for
	$z_K$ as differentiable function of $z_0$.
	For the derivative we have
	\begin{align*}
	\frac{d z_K }{dz_0}  &= 1 + \frac{\partial  R_z}{\partial x_0}  \frac{d x_0}{dz_0} +    \frac{\partial  R_z}{\partial y_K}  \frac{d y_K}{dz_K} \frac{d z_K}{d z_0} + \frac{\partial  R_z}{\partial z_0},
	\end{align*}
	with $R_z$ calculated in $(x_0(z_0),y_K (z_K),z_0)$.
	The tangent lines of the center leaves go to zero as $k\to \infty$. 
	The estimates in Lemma~\ref{l:n} imply that also $R_z$ and its derivatives converge to zero as $k\to\infty$, and thus 
	show that $z_K$ converges to the identity in $C^1$ as $k \to \infty$.
	
	Next, we proceed to rescale the $z$-coordinate so that it becomes defined on $J$, and fiber maps $f_\omega$
	converge to the identity map on $J$ as $k\to \infty$. 
	As the tangent lines of the center leaves vary continuously with the base point, the convergence is uniform in $\omega \in \Sigma$.
	We conclude that $F_k$ converges to $I$ in
	$\mathcal{C}^1 (\Sigma \times J)$ as $k \to \infty$.
\end{proof}

We note that as $k$ in  Proposition~\ref{p:skew} increases (and the neighbourhood $\mathcal{V}_k$ decreases), the number of iterates $K$ goes to infinity. Lemma~\ref{l:n} shows that this only facilitates the estimates for the return map
and leads to the convergence of $F_k$ to  $I$ as $k\to \infty$.

The following result implies that the center foliation is normally hyperbolic.

\begin{proposition}\label{p:strongstable}
	For sufficiently small  neighbourhoods $\mathcal{V}_k$ of $\{ p_a\}\cup \{q_a\}$
	as above, there exists a center stable lamination $\mathcal{F}^{sc}$ inside $\mathcal{V}_k$ 
	with $n-1$ dimensional leaves, containing the local center stable manifolds $W^{sc}_{loc} (\{p_a\})$
	and  $W^{sc}_{loc} (\{q_a\})$, so that 
	\begin{enumerate}[(i)]
	\item center stable leaves are foliated by local strong stable manifolds,
	\item the strong stable lamination $\mathcal{F}^s$ formed by the union of the strong stable manifolds is locally invariant, 
	\item
	the tangent spaces of the strong stable lamination depend continuously on the point.
	\end{enumerate}
	Corresponding statements also hold for a strong unstable lamination $\mathcal{F}^u$.
\end{proposition}

\begin{proof}
	We note first that a strong stable lamination is constructed by the same methodology used in the proof of Proposition~\ref{p:Vn} to construct a center lamination.
	In fact it suffices to replace $\mathcal{H}$ defined in \eqref{e:mathcalH} by
	\[
	\mathcal{H}^+: \mathcal{C} \to \mathcal{C}(\mathbb{N},\mathbb{R}^{2n-1})
	\] 
	given by the identity:
	if $\gamma_i = (x_i,y_i,z_i)$ and $\mathcal{H}^+ (\gamma) = \eta$ with 
	$\eta_i = (u_i,v_i,w_i)$, then
	\begin{align}
	\nonumber
	(u_{i+1},v_i,w_{i+1}) &=  
	\Phi^{-}_{\omega_i\omega_{i+1}} (x_i,y_{i+1},z_{i+1}), 
	\end{align}
	for $i \in \mathbb{N}$, and $u_0 = x_0\in \mathbb{R}^{n-1}$,  $w_0 = z_0 \in \mathbb{R}$.
	The center unstable lamination is constructed similarly, and the center lamination is in fact the intersection of the center stable lamination and the center unstable lamination.
	
	We continue with the strong unstable manifolds.
	One constructs a bundle of tangent spaces of strong unstable leaves and shows that these integrate to form
	a strong unstable foliation of center unstable leaves.
	Denote by $G^{n-1}(\mathbb{R}^{2n-1})$ the Grassmannian manifold
	of $(n-1)$ dimensional planes in $\mathbb{R}^{2n-1}$.
	A strong unstable foliation is  determined by its tangent bundle, hence by a section $\mathcal{F}^{cu} \to G^{n-1}(\mathbb{R}^{2n-1})$.
	Extend $\Psi$ to  $\Psi^{(1)}$ on the bundle
	$\mathbb{R}^{2n-1} \times G^{n-1}(\mathbb{R}^{2n-1})$ of $(n-1)$ dimensional 
	planes in $\mathbb{R}^{2n-1}$ over $\mathbb{R}^{2n-1}$ by
	\[
	\Psi^{(1)} (x, \alpha) = (\Psi(x), D\Psi(x)\alpha).
	\]
	The bundle $E^s$ of unstable directions over $\{p_a\}$ and over $\{q_a\}$
	is fixed under $\Psi^{(1)}$.

	Proposition~\ref{p:Vn} provides a strong unstable lamination $\mathcal{F}^{cu}$ on $\{ |x|\le 2, |y|\le 2, |z|\le 2\}$.
	A direct computation (compare \cite{hirpugshu77})
	shows that $\Psi^{(1)}$ is stable within the fibers
	$G^{n-1}(\mathbb{R}^{2n-1})$.
	Therefore iteration by $\Psi$ of a suitable trial foliation on 
	$\mathcal{F}^{cu}$ converges to an invariant
	unstable lamination. 	
\end{proof}	
\subsection*{Acknowledgements}
DT acknowledges support from Leverhulme Trust grant RPG 2021-072.
JSWL acknowledges support from JST Aihara Moonshot R \& D Grant Number JPMJMS2021 and 
EPSRC research grants EP/S023925/1 and EP/S515085/1.

\end{document}